\documentclass[ECP]{ejpecp} 

\newfont{\BB}{msbm10 scaled\magstep1}
\newfont{\bb}{msbm8}
\def\R{\mbox{\BB R}}
\def\RR{\mbox{\bb R}}

\def\D{\mbox{\BB D}}

\def\P{\mbox{\BB P}}
\def\E{\mbox{\BB E}}

\SHORTTITLE{Rate of change of \emph{varentropy}}

\TITLE{On the rate of change of \emph{Varentropy} for Markov diffusion processes\support{Supported
    by the NYUAD grant 76/71260/ADHPG/AD356/51119.}\
      } 

\AUTHORS{%
  Michele~Pavon\footnote{Mathematics Program, Division of Science,
 New York University, Abu Dhabi, UAE.
    \EMAIL{michele.pavon@nyu.edu}}}


\AMSSUBJ{60J60;62B10;94A17} 

\SUBMITTED{November12, 2023} 

\VOLUME{0}
\YEAR{2020}
\PAPERNUM{0}
\DOI{10.1214/YY-TN}

\ABSTRACT{Relying on the reverse-time space-time harmonic property of the ratio of two solutions of the Fokker-Plank equation, we establish an explicit formula for derivative of the \emph{varentropy} for a Markov diffusion process. The formula involves a nonlinear  function of the {\em local free energy} $\ln(p_t(x)/\bar{p}(x))$.We then verify that our formula yields the correct result in the simple case of a scalar Gaussian diffusion. In the latter case, varentropy is exponentially decaying to zero. }

\begin{document}



\section{Introduction}Entropy concentration phenomena have been long studied, see. e.g. \cite{Ja}. In \cite{FLM}, concentration of information for convex measures was studied through the variance of the {\em information content}, namely $-log p(X)$ where $p$ is the density of the probability distribution of the random vector $X$. In the \cite{S}, the cutoff phenomenon was explicitly connected to an {\em entropy concentration principle} for non-negatively curved Markov chains. The central indicator of the entropy concentration property is  the (worst-case) {\em Varentropy}, namely the variance of a relative information content with respect to the invariant measure. 

We study here the evolution of varentropy in the simple case of a Markov diffusion process with constant diffusion coefficient. Using a reverse-time space-time harmonic property of the ratio of two solutions of the Fokker-Plank equation, we obtain an explicit formula for the derivative of the varentropy which involves  a nonlinear  function of the {\em local free energy} $\ln(p_t(x)/\bar{p}(x))$, $\bar{p}$ being the invariant density. We then show that the formula applied to the scalar, Gaussian case yields the correct result.

The paper is outlined as follows. In Section \ref{finite energy}, we recall some results by F\"{o}llmer dating back to the 1980's. In Section \ref{Markov} we recall some reverse-time properties of the ratio of two solutions of a Fokker-Planck equation. In Section \ref{varentropy}, we present our main result Theorem \ref{var} on the time derivative of the Varentropy. In the final section, we show that the formula obtained in Section \ref{varentropy} yields the correct result when applied to a scalar, Gaussian diffusion.

\section{Finite-energy diffusions} \label{finite energy}
Let ${\cal W}_x$ be Wiener measure starting at $x$ on $\Omega=C([0,T],\R^n)$. Let 
\[{\cal W}=\int {\cal W}_x dx
\]
be {\em stationary Wiener measure}. We denote by $\mathcal D$ the family of probability measures on $\Omega$ which are equivalent to stationary Wiener measure. Let $\mathcal F_t$ and $\mathcal G_t$ be the $\sigma$-algebras of events observable up to time $t$ and from time $t$ on, respectively. By Girsanov's theorem, any $P\in\mathcal D$ has a forward drift $\beta_t^P$ measurable with respect to $\mathcal F_t$   and a backward drift $\gamma_t^P$ measurable with respect to $\mathcal G_t$. That means that, under $P$, the coordiante process admits for $0\le s<t\le T$  the representations
\begin{eqnarray}X_t-X_s&=&\int_s^t\beta_\tau^Pd\tau+W_t-W_s,\\X_t-X_s&=&\int_s^t\gamma_\tau^Pd\tau+\bar{W}_t-\bar{W}_s.
\end{eqnarray}
Here $W$ and $\bar{W}$ are standard $n$-dimensional Wiener processes adapted to $(\mathcal F_t)$ and to $(\mathcal G_t)$, respectively. Moreover, the two drifts satisfy the  condition
\[P\left[\int_0^T|\beta^P_t|^2dt\right]= P\left[\int_0^T|\gamma^P_t|^2dt\right]=1,
\]
see e.g \cite[Theorem 7.11]{LS}. By the Girsanov theory, we have $P$ a.s.
\begin{eqnarray}\nonumber
\frac{dQ}{dP}(X)&=&\frac{q_0(X_0)}{p_0(X_0)}\exp\left[\int_0^T(\beta^Q_t-\beta^P_t)dX_t+\frac{1}{2}\left(\beta^P_t-\beta^Q_t\right)^2dt\right]\\&=&\frac{q_T(X_T)}{p_T(X_T)}\exp\left[\int_0^T(\gamma^Q_t-\gamma^P_t)dX_t+\frac{1}{2}\left(\gamma^P_t-\gamma^Q_t\right)^2dt\right].\label{Girsanov}
\end{eqnarray}
\begin{proposition}\label{martingale}Let $P, Q\in\mathcal D$ have one-time density $p_t$ and $q_t$, $t\in [0,T]$, respectively. If $\gamma^P_t=\gamma^Q_t$ for $t\in [0,T]$, $\frac{q_t}{p_t}(X_t)$ is a $(P,\mathcal F_t)$ martingale on $[0,T]$. If $\beta^P_t=\beta^Q_t$ for $t\in [0,T]$, $\frac{q_t}{p_t}(X_t)$ is a reverse-time $(P,\mathcal G_t)$ martingale on $[0,T]$.
\end{proposition}
\begin{proof} Suppose $\gamma^P_t=\gamma^Q_t$ for $t\in [0,T]$. By (\ref{Girsanov}), under $P$,
\[\E\left\{\frac{dQ}{dP}| \mathcal F_t\right\}=\frac{dQ_{0t}}{dP_{0t}}=\frac{q_t}{p_t}(X_t).
\]
Being the conditional expectation of a fixed random variable with respect to an increasing family of $\sigma$-algebras, $\frac{q_t}{p_t}(X_t)$ is a supermantingale. Since it has constant expectation, it is actually a martingale. The other half is proven similarly.
\end{proof}

\section{Markov diffusion processes}\label{Markov}
Let $\{X_t; t\ge 0\}$ be a $\R^n$-valued Markov diffusion process defined on a probability space $(\Omega,\cal F,\P)$ with forward It$\hat{o}$ differential
\begin{equation}\label{FD}dX_t=b_+(t,X_t)dt+ \sigma dW_t,
\end{equation}
where $W$ is a standard $n$-dimensional Wiener process. Under rather weak assumptions \cite{N1,HP}, $X$ possesses on $t>0$ also a reverse-time differential
\begin{equation}\label{RD}dX_t=b_-(t,X_t)dt+ \sigma d\bar{W}_t.
\end{equation}
Here $\bar{W}$ is another $n$-dimensional standard Wiener process with $\bar{W}_s-\bar{W}_t$ independent of $X_\tau$ whenever $s<t\le \tau$. Here and in the sequel $dt>0$ so that, if $X$ is a finite-energy diffusion, i.e.  
\[\E\left\{\int_{t_0}^{t_1}\|b_+\|^2\right\}<\infty,
\]
we get the two drifts as Nelson's conditional derivatives, namely
\begin{eqnarray}\nonumber b_+(t,X_t)&=&\lim_{dt\searrow 0}E\left\{\frac{X_{t+dt}-X_t}{dt}|X_\tau, t_0\le \tau\le t\right\},\\\nonumber b_-(t,X_t)&=&\lim_{dt\searrow 0}E\left\{\frac{X_{t}-X_{t-dt}}{dt}|X_\tau, t\le \tau\le t_1\right\},
\end{eqnarray}
the limits being in $L^2_n(\Omega,\cal f,\P)$ and $dt>0$ for both formulas. Moreover, we have Nelson's duality formula
\begin{equation}\label{DF}b_-(t,X_t)=b_+(t,X_t)-\sigma^2\nabla\ln p_t(X_t) \:{\rm a.s.}
\end{equation}
where $p_t$ is the probability density of $X_t$ satisfying (at least weakly) the Fokker-Planck equation
\begin{equation}\label{FP}\frac{\partial p}{\partial t}+\nabla\cdot(b_+ p)-\frac{\sigma^2}{2}\Delta p=0.
\end{equation}
From (\ref{DF}), we also have
\begin{equation}\label{RFP}\frac{\partial p}{\partial t}+\nabla\cdot(b_- p)+\frac{\sigma^2}{2}\Delta p=0.
\end{equation}
There are two change of variables formulae related to (\ref{FD}) and (\ref{RD}).
Let $f:\R^n\times
[t_0,t_1] \rightarrow \R$ be twice continuously differentiable with respect to the spatial variable
and once with respect to time.
Then: 
\begin{eqnarray}
f(X_t,t)-f(X_s,s)=\int_s^t\left(\frac{\partial}{\partial
\tau}+b_+(\tau,X_{\tau})\cdot\nabla+\frac{\sigma^2}{2}\Delta\right)f(X_{\tau},\tau
)d\tau\\+\int_s^t\sigma\nabla f(X_\tau,\tau)\cdot
dW_\tau,\label{K5}\\f(X_t,t)-f(X_s,s)=
\int_s^t\left(\frac{\partial}{\partial
\tau}+b_-(\tau,X_{\tau})\tau)\cdot\nabla-\frac{\sigma^2}{2}\Delta\right)f(X_{\tau},\tau)d\tau
\\+\int_s^t\sigma\nabla f(X_{\tau})\cdot d\bar{W}_\tau. \label{K6} \end{eqnarray} The
stochastic integrals appearing in (\ref{K5}) and (\ref{K6}) are a (forward) Ito integral and a
backward Ito integral, respectively, see \cite {N3} for the details. Suppose that $b_+(t,x)=b_+(x)$ is such that (\ref{FP}) has $\bar{p}$ as invariant solution. Then, a direct calculation \cite{P}
shows that $\bar{p}/p_t$ is reverse-time space time harmonic, namely
\begin{equation}\label{RSTH}\left(\frac{\partial}{\partial
\tau}+b_-(\tau,x)\cdot\nabla-\frac{\sigma^2}{2}\Delta\right)\left(\frac{\bar{p}}{p_t}\right)=0.
\end{equation}
It follows that $M_t:=\frac{\bar{p}}{p_t}(X_t)$ is a reverse-time local martingale with respect to the decreasing family ${\cal F}_t^+(X)=\sigma\left\{X_\tau; \tau\ge t\right\}$. This follows also from Proposition \ref{martingale} for diffusions with diffusion coefficient $\sigma^2$. In view of (\ref{K6}), the reverse-time Ito differential of $M_t$ is 
\begin{equation}\label{DM} dM_t=\sigma\nabla\left(\frac{\bar{p}}{p_t}\right)(X_t)d\bar{W}_t.
\end{equation}

\section{Varentropy}\label{varentropy}

Define $-\ln p_t(x)$ as the {\em local entropy}. We mention that $-\log p_t$, where $p_t$ satisfies a Fokker-Planck equation, may be viewed as the value function of a stochastic control problem. In particular, in \cite[p.194]{P}, $-\log p(x,t)$ was named {\em local entropy} and various of its properties were established. The corresponding optimal control is  related to the so-called {\em score-function} $\nabla\log p(x,t)$ of generative models of machine learning based on flows.  Local entropy was recently rediscovered  in connection with an attempt to smooth the {\em energy landscape} of deep neural networks, see \cite{P2} and references therein. 

Let $\psi(t,x)=\ln(p_t/\bar{p})$ be the {\em local free energy} as in \cite{P}. Then 
\[\E\{\psi(t,X_t)\}=\D(p_t\|\bar{p})=\int_{\RR^n}\ln\frac{p_t}{\bar{p}}(x)p_t(x)dx,
\]
the {\em relative entropy} or {\em divergence} or {\em Kullback-Leibler index} between $p_t$ and $\bar{p}$.
Moreover, the free energy decays according to \cite{Gr}
\begin{equation}\label{FED}
\frac{d}{dt}\E\{\psi(t,X_t)\}=\frac{d}{dt}\D(p_t\|\bar{p})=-\frac{\sigma^2}{2}\int_{\RR^n}\|\nabla\ln\frac{p_t}{\bar{p}}(x)\|^2p_t(x)dx.
\end{equation}
Define the {\em Varentropy} ${\cal V}_{KL}(p_t\|\bar{p})$ as in \cite{S} 
\begin{eqnarray}\nonumber{\cal V}_{KL}(p_t\|\bar{p})&=&\int_{\RR^n}\left(\ln\frac{p_t(x)}{\bar{p}(x)}-\D(p_t\|\bar{p})\right)^2p_t(x)dx\\\nonumber&=&\E\left[\left(\psi(t,X_t)-\E(\psi(t,X_t)\right)^2\right]=\E\left\{\psi(t,X_t)^2\right\}-\E\{\psi(t,X_t)\}^2.
\end{eqnarray}
Hence, in view of (\ref{FED}), 
\begin{eqnarray}\nonumber\frac{d}{dt}{\cal V}_{KL}(p_t\|\bar{p})&=&\frac{d}{dt}\E\left\{\psi(t,X_t)^2\right\}-\frac{d}{dt}\left[\D(p_t\|\bar{p})^2\right]\\&=&\frac{d}{dt}\E\left\{\psi(t,X_t)^2\right\}+\D(p_t\|\bar{p})\sigma^2\int_{\RR^n}\|\nabla\ln\frac{p_t}{\bar{p}}(x)\|^2p_t(x)dx.\label{firstderVar}
\end{eqnarray}
To understand the rate of change of  the Varentropy, we need therefore to study 
\[\frac{d}{dt}\E\left\{\psi(t,X_t)^2\right\}=\frac{d}{dt}\E\left\{(-\ln M_t)^2\right\},
\]
where we recall that $M_t:=\frac{\bar{p}}{p_t}(X_t)$. 

\begin{theorem} \label{var}Let $\psi(t,x)=\ln(p_t/\bar{p})$ .The rate of change of the Varentropy is given by
\begin{equation}\frac{d}{dt}{\cal V}_{KL}(p_t\|\bar{p})=\sigma^2\E\left\{\left[-\psi(t,X_t)-1+\E\{\psi(t,X_t)\}\right]\|\nabla\psi(t,X_t)\|^2\right\}.\label{var_increase_decay}
\end{equation}
\end{theorem}

\begin{proof}
We first compute the reverse-time differential of $N_t = (-\ln M_t)^2=(\ln M_t)^2=\left[\ln\frac{p_t}{\bar{p}}(X_t)\right]^2=\psi(t,X_t)^2$. Note that $f(x)=(\ln x)^2$ is not convex nor concave on $x>0$ and we have
\[f'(x)=2\ln x\cdot\frac{1}{x}, \quad f''(x)=\frac{2(1-\ln x)}{x^2}.
\] 
By (\ref{DM}) and the generalization of (\ref{K6}), we get
\begin{eqnarray}\nonumber dN_t&=&\frac{df}{dx}(M_t)dM_t-\frac{1}{2}\frac{d^2f}{dx^2}(M_t)d\langle M\rangle_t=-\frac{1-\ln M_t}{M_t^2}\sigma^2\left\|\nabla\frac{p_t}{\bar{p}}\right\|^2dt+2\ln M_t\frac{1}{M_t}\nabla \frac{\bar{p}}{p_t}\sigma d\bar{W}_t\\&=&\sigma^2\left(-\ln \frac{p_t}{\bar{p}}(X_t)-1\right)\|\nabla\ln\frac{p_t}{\bar{p}}(X_t)\|^2dt+2\ln \frac{p_t}{\bar{p}}(X_t)\nabla\ln\frac{p_t}{\bar{p}}(X_t)\sigma d\bar{W}_t.\label{RDN}
\end{eqnarray}
By a locatization argument, the stochastic integral
\[\int_{t_0}^{t_1}2\ln \frac{p_t}{\bar{p}}(X_t)\nabla\ln\frac{p_t}{\bar{p}}(X_t)\sigma d\bar{W}_t
\]
has zero expectation. Thus, in view of (\ref{firstderVar}) and (\ref{RDN}), we get
\begin{eqnarray}\nonumber \frac{d}{dt}{\cal V}_{KL}(p_t\|\bar{p})&=& \frac{d}{dt}\E\left[\left(\psi(t,X_t)-\E(\psi(t,X_t)\right)^2\right]\nonumber\\\nonumber&=&\frac{d}{dt}\E\left\{N_t\right\}+\D(p_t\|\bar{p})\sigma^2\int_{\RR^n}\|\nabla\ln\frac{p_t}{\bar{p}}(x)\|^2p_t(x)dx\\&=&\sigma^2\E\left\{\left(-\ln\frac{p_t}{\bar{p}}(X_t)-1\right)\|\nabla\ln\frac{p_t}{\bar{p}}(X_t)\|^2\right\}+\D(p_t\|\bar{p})\sigma^2\int_{\RR^n}\|\nabla\ln\frac{p_t}{\bar{p}}(x)\|^2p_t(x)dx\nonumber\\&=&\sigma^2\E\left\{\left(-\ln\frac{p_t}{\bar{p}}(t,X_t)-1+\D(p_t\|\bar{p})\right)\|\nabla\ln\frac{p_t}{\bar{p}}(X_t)\|^2\right\}\nonumber\\&=&\sigma^2\E\left\{\left[-\psi(t,X_t)-1+\E\{\psi(t,X_t)\}\right]\|\nabla\psi(t,X_t)\|^2\right\}.\label{var_increase_decay}
\end{eqnarray}
\end{proof}

\section{The scalar Gaussian case}\label{Gaussian}
Consider the case of a scalar, zero-mean, Gaussian diffusion $\{X_t; t\ge 0\}$ with forward It$\hat{o}$ differential
\[dX_t=-\frac{1}{2} X_tdt +dW_t
\]
and initial condition $X_0$ distributed according to
\[p_0(x)=\frac{1}{\sigma_0\sqrt{2\pi}}\exp\left[-\frac{x^2}{\sigma_0^2}\right].
\]
In this simple case, the time derivative of Varentropy can be computed explicitly.
The variance $\sigma_t^2$ of $X_t$ satisfies
\[\frac{d}{dt}(\sigma_t^2)=-\sigma_t^2+1, \quad \lim_{t\rightarrow\infty}\sigma^2_t=\bar{\sigma}^2=1.
\]
Thus the invariant density is simply the standard Normal
\[\bar{p}(x)=\frac{1}{\sqrt{2\pi}}\exp\left[-x^2\right].
\]
Then
\[\psi(t,x)=\ln(p_t/\bar{p})(x)=\ln\left\{\frac{1}{\sigma_t}\exp\left[-\frac{x^2}{2}\left(-1+\frac{1}{\sigma_t^2}\right)\right]\right\}=-\ln\sigma_t-\frac{x^2}{2}\left(\frac{1-\sigma_t^2}{\sigma_t^2}\right)
\]
It follows that
\[\E\{\psi(t,X_t)\}=-\ln\sigma_t-\frac{1}{2}\left(1-\sigma_t^2\right).
\]
Observe that
\[\dot{\sigma}_t=\frac{d}{dt}\left(\sqrt{\sigma_t^2}\right)=\frac{1}{2}\frac{1}{\sqrt{\sigma_t^2}}\frac{d}{dt}\left(\sigma^2_t\right)=\frac{1}{2\sigma_t}\left(1-\sigma_t^2\right).
\]
Then
\[\frac{d}{dt}\E\{\psi(t,X_t)\}=-\frac{\dot{\sigma}_t}{\sigma_t}+\frac{1}{2}\frac{d}{dt}\left(\sigma_t^2\right)=-\frac{1}{2\sigma_t^2}\left(1-\sigma_t^2\right)+\frac{1}{2}\left(1-\sigma_t^2\right)=-\frac{1}{2}\frac{\left(1-\sigma_t^2\right)^2}{\sigma_t^2}.
\]
Recal that the Varentropy  is
\[{\cal V}_{KL}(p_t\|\bar{p})=\E\left\{\psi(t,X_t)^2\right\}-\E\{\psi(t,X_t)\}^2.
\]
We compute first $\E\left\{\psi(t,X_t)^2\right\}$. We have
\[\psi(x,t)^2=\left[-\ln\sigma_t-\frac{x^2}{2}\left(\frac{1-\sigma_t^2}{\sigma_t^2}\right)\right]^2=\left(\ln \sigma_t\right)^2+\ln\sigma_t x^2\left(\frac{1-\sigma_t^2}{\sigma_t^2}\right)+\frac{x^4}{4}\left(\frac{1-\sigma_t^2}{\sigma_t^2}\right)^2.
\]
As $X_t$  ia a centered Gaussian $\E\{X_t^4\}=3\sigma_t^4$. Hence, we get
\[\E\left\{\psi(x,t)^2\right\}=\left(\ln \sigma_t\right)^2+\ln\sigma_t \left(1-\sigma_t^2\right)+\frac{3}{4}\left(1-\sigma_t^2\right)^2.
\]
Thus,
\begin{eqnarray}\nonumber{\cal V}_{KL}(p_t\|\bar{p})&=&\left(\ln \sigma_t\right)^2+\ln\sigma_t \left(1-\sigma_t^2\right)+\frac{3}{4}\left(1-\sigma_t^2\right)^2-\left[\ln\sigma_t+\frac{1}{2}\left(1-\sigma_t^2\right)\right]^2\\&=&\left(\ln \sigma_t\right)^2+\ln\sigma_t \left(1-\sigma_t^2\right) +\frac{3}{4}\left(1-\sigma_t^2\right)^2-\left(\ln \sigma_t\right)^2-\ln\sigma_t \left(1-\sigma_t^2\right)-\frac{1}{4}\left(1-\sigma_t^2\right)^2\nonumber\\&=&\frac{1}{2}\left(1-\sigma_t^2\right)^2.\nonumber
\end{eqnarray}
It follows that
\begin{equation}\label{varentropydecay}\frac{d}{dt}{\cal V}_{KL}(p_t\|\bar{p})=-\left(1-\sigma_t^2\right)\frac{d}{dt}\left(\sigma_t^2\right)=-\left(1-\sigma_t^2\right)^2=-2{\cal V}_{KL}(p_t\|\bar{p}).
\end{equation}
Hence, Varentropy is exponentially decaying according to
\[{\cal V}_{KL}(p_t\|\bar{p})={\cal V}_{KL}(p_0\|\bar{p})e^{-2t}=\frac{1}{2}\left(1-\sigma_0^2\right)^2e^{-2t}.
\]
We  now try to compute the Varentropy time derivative using formula (\ref{var_increase_decay})
\[\frac{d}{dt}{\cal V}_{KL}(p_t\|\bar{p})=\E\left\{\left[-\psi(t,X_t)-1+\E\{\psi(t,X_t)\}\right]\left\|\frac{\partial \psi(t,X_t)}{\partial x}\right\|^2\right\}.
\]
\begin{eqnarray}\nonumber\frac{d}{dt}{\cal V}_{KL}(p_t\|\bar{p})&=&\E\left\{\left[\ln\sigma_t+\frac{X_t^2}{2}\left(\frac{1-\sigma_t^2}{\sigma_t^2}\right)-1-\ln\sigma_t-\frac{1}{2}\left(1-\sigma_t^2\right)\right]\cdot\left[X_t^2\left(\frac{1-\sigma_t^2}{\sigma_t^2}\right)^2\right]\right\}\\\nonumber&=&\E\left\{\frac{X_t^4}{2}\left(\frac{1-\sigma_t^2}{\sigma_t^2}\right)^3\right\}-\E\left\{X_t^2\left(\frac{1-\sigma_t^2}{\sigma_t^2}\right)^2\right\}-\frac{1}{2}\E\left\{X_t^2\left(1-\sigma_t^2\right)\left(\frac{1-\sigma_t^2}{\sigma_t^2}\right)^2\right\}\\\nonumber&=&\frac{3\sigma^4_t}{2}\left(\frac{1-\sigma_t^2}{\sigma_t^2}\right)^3-\frac{\left(1-\sigma_t^2\right)^2}{\sigma_t^2}-\frac{1}{2}\frac{\left(1-\sigma_t^2\right)^3}{\sigma_t^2}\\\nonumber&=&\frac{3}{2\sigma_t^2} \left(1-\sigma_t^2\right)^3-\frac{\left(1-\sigma_t^2\right)^2}{\sigma_t^2}-\frac{1}{2}\frac{\left(1-\sigma_t^2\right)^3}{\sigma_t^2}\\\nonumber&=&\frac{\left(1-\sigma_t^2\right)^3}{\sigma_t^2}-\frac{\left(1-\sigma_t^2\right)^2}{\sigma_t^2}\\\nonumber&=&-\left(1-\sigma_t^2\right)^2
\end{eqnarray}
which coincides with (\ref{varentropydecay}).
\section{Closing comments}
Further study, possibly based on functional inequalities, is needed to establish if (\ref{var_increase_decay}) implies that Varentropy ${\cal V}_{KL}(p_t\|\bar{p})$, in analogy to relative entropy $\D(p_t\|\bar{p})$, is always decreasing for Markov diffusions like in the Gaussian case. If not, there could be cases where Varentropy  attains  a maximum at some time $t_M$ after which it decreases to zero in analogy to what is observed in some Markov chains models \cite{S}.


\begin{thebibliography}{99}

\bibitem{F} F\"{o}llmer, H.: Time reversal on Wiener space, in: {\em Stochastic Processes -
Mathematics and Physics} , Lect. Notes in Math. 1158, Springer-Verlag, New York,1986, p. 119. MR0838561
\bibitem{F2} F\"{o}llmer, H.: Random fields and diffusion processes, in:
{\em \`Ecole d'\`Et\`e de Probabilit\`es de Saint-Flour XV-XVII},
edited by P. L.
Hennequin, Lecture Notes in Mathematics, Springer-Verlag, New York,
1988, vol.1362,102-203. MR0983373
\bibitem{FLM}Fradelizi, M., Li, J., Madiman, M.: Concentration of information content for convex measures, {\em Electron. J. Probab.} {\bf 25}, Paper No. 20, 22, 2020. MR4073681
\bibitem{Gr} Graham, R.: Path integral methods in nonequilibrium
thermodynamics and statistics, in {\it Stochastic Processes in
Nonequilibrium Systems}, L. Garrido, P. Seglar and P.J.Shepherd Eds.,
Lecture Notes in Physics 84, Springer-Verlag, New York, 1978, 82-138. MR0532361
\bibitem {HP} Haussmann, U. G. and Pardoux, E.: Time reversal of diffusions, {\em The Annals of Probability}, 1188-1205, 1986.  MR0866342
\bibitem{Ja} Jaynes, E. T.:  {\em Papers on Probability, Statistics and
Statistical Physics}, R.D. Rosenkranz ed., Dordrecht, 1983. MR0786828
\bibitem{LS} Liptser R. S. and Shiryaev, A. N.: {\em Statistics of Random Processes I, General Theory}, Springer-Verlag, 1977. MR0474486
\bibitem{N1}Nelson, E.: {\em Dynamical Theories of Brownian
Motion}, Princeton University Press, Princeton, 1967.  MR0214150 
\bibitem{N3} Nelson, E.: Stochastic mechanics and random fields, in{\em \`Ecole d'\`Et\`e de Probabilit\`es de
Saint-Flour XV-XVII}, edited by P. L. Hennequin, Lecture Notes in Mathematics, Springer-Verlag,
New York, 1988, Vol.1362, pp. 428-450. MR0983376 
\bibitem {P} Pavon, M.: Stochastic control and nonequilibrium thermodynamical systems, {\em Appl. 
Math. and Optimiz.} {\bf 19} 1989, 187-202.  MR0962891
\bibitem{P2} Pavon, M.: On Local Entropy, Stochastic Control, and Deep Neural Networks, {\em IEEE Control Systems Letters}, {\bf 7}, 2022, 437-441.  MR4523720
\bibitem{S} Salez, J.: Cutoff for non-negatively curved Markov chains, ArXiv e-prints, arXiv: 2102.05597v2.
\end{thebibliography}
\end{document}